\documentclass{amsproc}[12pt]

\usepackage[ansinew]{inputenc}
\usepackage{graphicx}
\usepackage{color}

\usepackage{enumerate,latexsym}
\usepackage{latexsym}
\usepackage{amsmath,amssymb}
\usepackage{graphicx}
\newfont{\bb}{msbm10 at 11pt}
\newfont{\bbsmall}{msbm8 at 8pt}

\newtheorem{teorema}{Theorem}

\newtheorem{lema}{Lemma}
\newtheorem{corolario}{Corollary}

\newtheorem{observacion}{Remark}
\newtheorem{nota}{Remark}

\newtheorem{claim}{Claim}

\def\z{\hbox{\bb Z}}

\newcommand{\ep}{\varepsilon}

\def\cA{\mathcal{A}}

\def\cK{\mathcal{K}}

\def\cL{\mathcal{L}}
\def\cU{\mathcal{U}}
\def\cM{\mathcal{M}}

\def\cH{\mathcal{H}}

\let\8=\infty \let\0=\emptyset  
\let\hat=\widehat

\let\tilde=\widetilde

\let\landa=\lambda
\let\alfa=\alpha

\let\parc=\partial
\let\ep=\varepsilon

\def\landa{\lambda}

\def\flecha{\rightarrow}
\def\esiz{\langle}
\def\esde{\rangle}

\def\cte.{\mathop{\rm cte.}\nolimits}

\def\cosh{\mathop{\rm cosh }\nolimits}

\def\A{\mathbb{A}}
\def\Z{\mathbb{Z}}

\def\R{\mathbb{R}}

\def\C{\mathbb{C}}
\def\D{\mathbb{D}}

\def\S{\mathbb{S}}

\newcommand{\beq}{\begin{equation}}
\newcommand{\eeq}{\end{equation}}

\numberwithin{equation}{section}

\begin{document}

\begin{title}[Isolated singularities of elliptic Monge-Ampère
equations]{A classification of isolated singularities of elliptic
Monge-Ampère equations in dimension two}
\end{title}
\today
\author{José A. Gálvez}
\address{José A. Gálvez, Departamento de Geometría y Topología,
Universidad de Granada, 18071 Granada, Spain}
 \email{jagalvez@ugr.es}

\author{Asun Jiménez}
\address{Asun Jiménez, Departamento de Geometria, IME,
Universidade Federal Fluminense, 24.020-140  Niterói, Brazil}
 \email{asunjg@vm.uff.br}

\author{Pablo Mira}
\address{Pablo Mira, Departamento de Matemática Aplicada y Estadística, Universidad Politécnica de
Cartagena, 30203 Cartagena, Murcia, Spain.}

\email{pablo.mira@upct.es}

\thanks{The authors were partially supported by
MICINN-FEDER, Grant No. MTM2010- 19821, Junta de Andalucía Grant No.
FQM325, the Programme in Support of Excellence Groups of Murcia, by
Fundación Séneca, R.A.S.T 2007-2010, reference 04540/GERM/06 and
Junta de Andalucía, reference P06-FQM-01642.}

\subjclass{35J96, 53C42}


\keywords{Monge-Ampère equation, isolated singularity, prescribed
curvature, boundary regularity.}

\begin{abstract}
Let $\cM_1$ denote the space of solutions $z(x,y)$ to an elliptic,
real analytic Monge-Ampère equation ${\rm det} (D^2 z
)=\varphi(x,y,z,Dz)>0$ whose graphs have a non-removable isolated
singularity at the origin. We prove that $\cM_1$ is in one-to-one
correspondence with $\cM_2\times \Z_2$, where $\cM_2$ is a suitable
subset of the class of regular, real analytic strictly convex Jordan
curves in $\R^2$. We also describe the asymptotic behavior of
solutions of the Monge-Ampère equation in the $C^k$-smooth case, and
a general existence theorem for isolated singularities of analytic
solutions of the more general equation ${\rm det} (D^2 z
+\mathcal{A}(x,y,z,Dz))=\varphi(x,y,z,Dz)>0$.
\end{abstract}

\maketitle

\section{Introduction}

In 1955 K. Jörgens wrote a seminal paper \cite{Jor} which initiated
the study of isolated singularities of the the classical elliptic
Monge-Ampère equation in dimension two,

 \begin{equation}\label{purema}
{\rm det} (D^2 z)=\varphi (x,y,z,Dz)>0,
 \end{equation}
where $D,D^2$ denote the gradient and Hessian operators. Jörgens
proved for $\varphi=1$ a removable singularity theorem, and gave a
description of the behavior of a solution to ${\rm det} (D^2 z)=1$
around a non-removable isolated singularity.

In this paper we classify the isolated singularities of
\eqref{purema} in the case that $\varphi$ is real analytic, and give
a complete description of the asymptotic behavior of such solutions
around an isolated singularity when $\varphi$ is only of class
$C^k$. Specifically, we give this classification by explicitly
parametrizing the moduli space of solutions to \eqref{purema} with a
non-removable isolated singularity at some given point, as we
explain next.

By convexity, any solution $z$ to \eqref{purema} defined on a
punctured disk extends continuously to the puncture. If the
extension is not $C^2$, we say that the puncture is an
\emph{isolated singularity}. The gradient of $z$ converges at the
singularity to a point, a segment, or a closed convex curve. We call
this set the \emph{limit gradient} of $z$ at the singularity. The
limit gradient $\gamma\subset \R^2$ describes the asymptotic
behavior of $z(x,y)$ as $(x,y)$ converges to the puncture.

Without loss of generality, we will assume that the singularity is
placed at $(0,0,0)$. We let $\varphi>0$ be defined on an open set
$\cU\subset \R^5$ such that $\cH:=\{(p,q)\in \R^2 : (0,0,0,p,q)\in
\cU\}\neq \emptyset$. Note that if $z$ is a solution to
\eqref{purema} with an isolated singularity at the origin, by
continuity the limit gradient of $z$ at the origin is contained in
$\overline{\cH}\subset \R^2$. As a matter of fact, $\gamma$ is
contained in $\cH$ in many natural situations; for instance if
$\cH=\R^2$ or, more generally, if $\cU\subset \R^3\times \cH$ where
$\cH\subset \R^2$ is simply connected (see Remark \ref{remhdef}).

With these conventions, we prove:

\begin{teorema}\label{main}
Let $\varphi\in C^{\omega}(\cU)$, $\varphi>0$. Let $\cM_1$ denote
the class of solutions $z$ to \eqref{purema} that have an isolated
singularity at the origin, and whose limit gradient at the
singularity is contained in $\cH\subset \R^2$. Let $\cM_2$ denote
the class of regular, analytic, strictly convex Jordan curves
$\gamma$ in $\cH\subset \R^2$.

Then, the map sending each $z\in \cM_1$ to $(\gamma,\ep)$, where
$\gamma$ is its limit gradient at the singularity and $\ep\in \Z_2$
is $0$ (resp. $1$) if $z_{xx}>0$ (resp. $z_{xx}<0$) defines a
bijective correspondence between $\cM_1$ and $\cM_2\times \Z_2$.
\end{teorema}

Theorem \ref{main} is a consequence of two more general results that
we obtain.

In Theorem \ref{mainth22} we prove that if $\varphi \in C^k(\cU)$,
$k\geq 4$, the limit gradient of any solution to \eqref{purema} with
an isolated singularity at the origin is a $C^{k,\alfa}$-smooth
regular, strictly convex (i.e. of nowhere-zero curvature) Jordan
curve $\gamma$ in $\R^2$. If $\varphi$ is analytic, we show that
$\gamma$ is also analytic. We also provide a way of parameterizing
the graph of any such solution so that the resulting map is defined
on an annulus and extends smoothly (or analytically) across the
boundary circle which the parametrization collapses to the
singularity. This result provides a complete description of the
asymptotic behavior of a solution to \eqref{purema} with $\varphi\in
C^k (\cU)$, $k\geq 4$, at an isolated singularity.

In Theorem \ref{mainth11} we give a general existence theorem for
isolated singularities of the \emph{general elliptic equation of
Monge-Ampère type} in dimesion two,
 \begin{equation}\label{eq0}
{\rm det }( D^2 z + \mathcal{A}(x,y,z,Dz)) =\varphi(x,y,z,D z)>0.
 \end{equation}
We prove that if both $\varphi$ and the symmetric matrix $\cA$ are
analytic in $\cU\subset \R^5$, any regular, analytic, strictly
convex Jordan curve $\gamma\in \cH$ can be realized as the limit
gradient of a solution to \eqref{eq0} which has an isolated
singularity at the origin. Even more generally, if in our
construction process we start with a closed, analytic curve in $\cH$
(not necessarily convex or regular), we obtain a \emph{multivalued
solution} to \eqref{eq0} with a singularity at the origin; see also
\cite{CaLi} for a study of multivalued solutions of Monge-Ampère
equations.

Theorem \ref{main} is not true for the more general equation
\eqref{eq0}; see Remark \ref{remaruno}.

The study of isolated singularities is a fundamental problem in the
theory of nonlinear geometric PDEs, and has been extensively
studied. Several of such elliptic equations (including the minimal
surface equation \cite{Ber}) only admit removable isolated
singularities; see \cite{LeRo} and references therein. The
asymptotic behavior at an isolated singularity of solutions to fully
nonlinear conformally invariant geometric PDEs has been studied in
detail in many works, see for instance
\cite{CHY,Gon,GuVi,HLT,Li,LiNg} and references therein (see also
\cite{CGS,KMPS,MaPa}). Some previous works on isolated singularities
of elliptic Monge-Ampère equations can be found in
\cite{ACG,Bey1,Bey2,GaMi,GHM,GMM,HeB,JiXi,Jor,ScWa}.

The apparition of solutions to \eqref{purema} with non-removable
isolated singularities is a very natural phenomenon, see Figure 1.
This justifies the interest of the study of the asymptotic behavior
and classification of such isolated singularities beyond a
\emph{removable singularity} type theorem. It should be emphasized
that, typically, solutions to \eqref{purema} with non-removable
isolated singularities do not belong to the usual classes of
generalized solutions to \eqref{purema} (viscosity solutions,
Alexandrov solutions).

\begin{figure}[h]
  \begin{center}
    \begin{tabular}{cc}
    \includegraphics[clip,width=5.5cm]{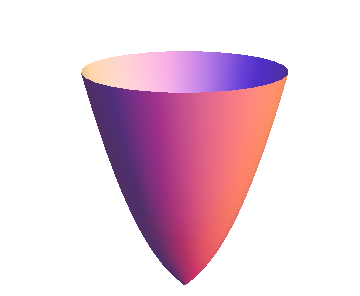} &
    \includegraphics[clip,width=3.3cm]{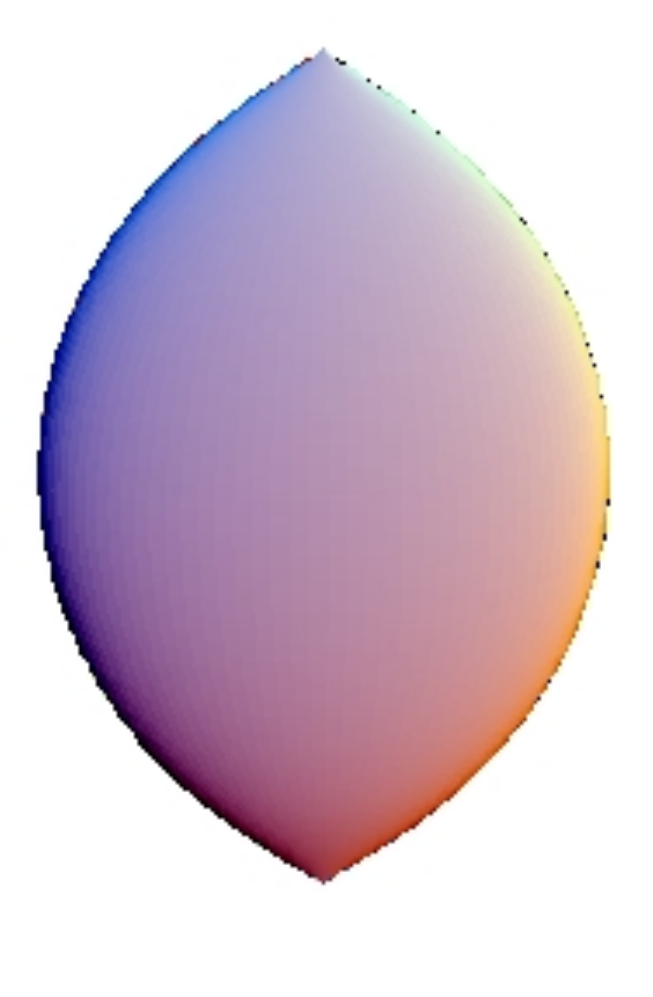}
\end{tabular}
\caption{Left: the radial function $z (x,y)= \frac{1}{2}\left(
 \sqrt{r^2(1+r^2)} + \sinh^{-1} (r)\right)$, $
r=\sqrt{x^2+y^2},$ is the simplest solution to  ${\rm det} (D^2
z)=1$ other than quadratic polynomials. Right: a rotational peaked
sphere in $\R^3$; it is the simplest $K=1$ surface in $\R^3$ other
than round spheres. Both examples present non-removable isolated
singularities.} \end{center}
 \end{figure}

Let us note that some of the arguments that we use here seem
specific of the two-dimensional case, since they rely on complex
analysis and surface theory. Nonetheless, the basic strategy of our
classification --transforming the PDE into a first order
differential system for which the isolated singularity turns into a
regular boundary curve, and then study the Cauchy problem for that
system along the boundary to determine the asymptotic behavior at
the singularity-- seems applicable to other fully nonlinear PDEs
admitting cone singularities, even in arbitrary dimension.

The PDEs \eqref{purema} and \eqref{eq0} appear in a variety of
applications, among which we may quote optimal transport problems,
isometric embedding of abstract Riemannian metrics, surfaces of
prescribed curvature in Riemannian and Lorentzian three-manifolds,
parabolic affine spheres, linear Weingarten surfaces, etc. In this
way, the results of this paper frequently admit reformulations in
these specific theories. For instance, some solutions with cone
singularities that we construct here provide (via the Legendre
transform) solutions to an obstacle problem for Monge-Ampère
equations, as explained in \cite{Sav}.

We shall give in an Appendix, as an example of a geometric
application of this type, a classification of the isolated
singularities of embedded surfaces in $\R^3$ with prescribed
positive Gaussian curvature.

\section{Asymptotic behavior at the singularity}\label{sec: unicidad}

In this section we will use the following conventions:

 \begin{itemize}
   \item
$\Omega=\{(x,y)\in\R^2 : 0<x^2+y^2<\rho^2\} $, a punctured disc
centered at the origin.
 \item
$\cU\subset \R^5$ is an open set such that $$\cH:=\{(p,q)\in \R^2 :
(0,0,0,p,q)\in \cU\}\neq \emptyset.$$
 \item
$\varphi\in C^k (\cU)$, $\varphi>0$, $k\geq 4$.
 \item
$z$ is a solution to \eqref{purema} in $\Omega$ with an isolated
singularity at $(0,0)$, which is of class $C^{k+1,\alfa}$ on compact
sets of $\Omega$. \emph{We will assume without loss of generality
from now on that $z$ has been continuously extended to the origin by
$z(0,0)=0$}.
 \item
$\gamma\subset \R^2$ will denote the \emph{limit gradient} of $z$ at
the origin; that is, $\gamma\subset \R^2$ is the set of points
$\xi\in \R^2$ for which there is a sequence $q_n\to (0,0)$ in
$\Omega$ such that $(z_x,z_y)(q_n)\to \xi$. Note that $\gamma\subset
\overline{\cH}$.
 \item
We will assume without loss of generality that $z_{xx}>0$. Observe
that by \eqref{purema} and since $\Omega$ is connected, either
$z_{xx}>0$ or $z_{xx}<0$ on $\Omega$. If $z$ is a solution to
\eqref{purema} with $z_{xx}<0$ and an isolated singularity at the
origin, then $\tilde{z}(x,y):=-z(-x,-y)$ is a solution to $z_{xx}
z_{yy}-z_{xy}^2=\tilde{\varphi} (x,y,z,z_x,z_y)$, where
$\tilde{\varphi} (x,y,z,p,q)=\varphi(-x,-y,-z,p,q)$, with
$\tilde{z}_{xx}>0$ and an isolated singularity at the origin. Note
that the limit gradients $\gamma,\tilde{\gamma}$ of $z$ and
$\tilde{z}$ at the origin coincide, and that $\cH=\tilde{\cH}$.
 \item
$\gamma\subset \cH$. See the next remark.
 \end{itemize}

\begin{observacion}\label{remhdef}
The condition that $\gamma\subset \cH$ (and not just that
$\gamma\subset \overline{\cH}$, which is always true by continuity)
automatically holds if the domain $\cU\subset \R^5$ where $\varphi$
is defined and positive has a simple geometry. Indeed, observe that
$z\in C^2(\Omega)$ is a locally strictly convex graph, as well as a
continuous convex graph on the convex planar set
$\Omega\cup\{(0,0)\}$. It is then easy to see from the theory of
convex sets that given a Jordan curve $\Gamma\subset\Omega$, then
$\hat{\Gamma}=(z_x,z_y)|_{\Gamma}$ is a Jordan curve in $\R^2$ with
the property that if $(x,y)\in\Omega$ is in the interior of the
bounded domain determined by $\Gamma$, then $(z_x(x,y),z_y(x,y))$ is
contained in the interior of the bounded domain determined by
$\hat{\Gamma}$. Hence, it is clear that $\gamma\subset {\mathcal H}$
if, for instance, ${\mathcal H}$ is simply connected and ${\mathcal
U}\subset \R^3\times{\mathcal H}$.
\end{observacion}

In the above conditions, the expression \beq\label{metricapura}
ds^2=z_{xx}dx^2+2z_{xy}dxdy+z_{yy}dy^2 \eeq is a Riemannian metric
on $\Omega$. It is a well known fact that $ds^2$ admits conformal
coordinates $w:=u+iv$ such that

\beq\label{conforme2}ds^2=\frac{\sqrt{\varphi}}{u_xv_y-u_yv_x}|dw|^2.
\eeq That is, there exists a $C^2$-diffeomorphism

\begin{equation}\label{tristar}
\Phi:\Omega\rightarrow  \Lambda:=\Phi(\Omega)\subset \R^2 ,\qquad
(x,y)\mapsto \Phi(x,y)=(u(x,y),v(x,y))
\end{equation}
satisfying \beq x_uy_v-x_vy_u>0,\eeq and the Beltrami system
 \beq\label{sist}  \left(\begin{array}{c}
 v_x\\v_y
 \end{array}\right)=\frac{1}{\sqrt{\varphi}}\left(\begin{array}{cc}z_{xy} & -z_{xx}\\z_{yy} & -z_{xy}
 \end{array}\right)\left(\begin{array}{c}
 u_x\\u_y\end{array}\right).\eeq

Here, $\Lambda$ is a domain in $\R^2\equiv \C$ which is conformally
equivalent to either the punctured disc $\D^*$ or an annulus
$\A_{\varrho}=\{\zeta\in \C : 1<|\zeta|<\varrho\}$. As $z$ does not
extend smoothly across the origin, by \cite[Lemma 3.3]{HeB} we have:
  \begin{lema}\label{anillo}
$\Lambda$ is conformally equivalent to some annulus $\A_{\varrho}$.
  \end{lema}
Thus, in order to study solutions of \eqref{purema} with an isolated
singularity, we may assume $\Lambda=\A_{\varrho}$. If we denote
$\Sigma_R :=\{w: 0< {\rm Im} (w)<R\}$, then $\A_{\varrho}$ is
conformally equivalent to $\Gamma_R:= \Sigma_R /(2\pi \Z)$ for $R=
\log \varrho$. So, composing with this conformal equivalence we will
suppose that the map $\Phi$ in \eqref{tristar} is a diffeomorphism
from $\Omega$ into $\Gamma_R$; in particular, $\Phi$ is
$2\pi$-periodic and $(u,v)$ will denote the canonical coordinates of
the strip $\Sigma_R$.

Let $G=\{(x,y,z(x,y)):(x,y)\in \Omega\}\subset \R^3$ be the graph of
$z(x,y)$. By using the parameters $(u,v)$, we may parameterize $G$
as a map

 \begin{equation}\label{grafconf}
\psi (u,v)=(x(u,v),y(u,v),z(u,v)):\Gamma_R\flecha G\subset \R^3
  \end{equation}
such that $\psi$ extends continuously to $\R$ with
$\psi(u,0)=(0,0,0)$.

In this section we prove the following result about the asymptotic
behavior, parametrization, and uniqueness in terms of the limit
gradient of solutions to \eqref{purema} at an isolated singularity.

\begin{teorema}\label{mainth22}
In the previous conditions, assume that $\varphi\in C^{k}(\cU)$,
$k\geq 4$ (resp. $\varphi\in C^{\omega} (\cU)$). Then:
\begin{enumerate}
   \item
$\gamma$ is a regular, strictly convex Jordan curve in $\R^2$, which
is $C^{k,\alpha}$ $\forall\alpha\in(0,1)$ (resp. analytic).
 \item
If $(u,v)$ denote conformal coordinates on $\Sigma_R$ for the metric
$ds^2$ as explained previously, and $p=z_x$, $q=z_y$ are viewed as
functions of $(u,v)$, then those functions extend
$C^{k,\alpha}$-smoothly (resp. analytically) to $\Sigma_R\cup \R$
and $\gamma(u):=(p(u,0),q(u,0))$ is a $C^{k,\alpha}$ (resp.
analytic), $2\pi$-periodic, negatively oriented parametrization of
$\gamma$ such that $\gamma'(u)\neq(0,0)$ for all $u\in\R$.
 \item
Let $\varphi\in C^{\omega}(\cU)$, and consider $z,z'\in
C^{\omega}(\Omega)$ two solutions to \eqref{purema} with an isolated
singularity at $(0,0)$ and $z(0,0)=z'(0,0)=0$, with the same limit
gradient $\gamma\subset \cH$ at the origin, and such that both
$z_{xx}$ and $z'_{xx}$ are positive. Then the graphs of $z$ and $z'$
agree on an open set containing the origin.
 \end{enumerate}
\end{teorema}
\begin{proof}
From now on, we will consider all the functions depending on the
parameters $(u,v)$ via $(x,y)=\Phi^{-1}(u,v)$. For simplicity, we
keep the same notation. From system \eqref{sist} (see for example
\cite{Bey1}) we have the following equations:
 \beq\label{d1p}
(p_u,p_v,q_u,q _v)=\sqrt{ \varphi }(y_v,-y_u,-x_v,x_u) \eeq Moreover
we have that

  $$z_v=px_v+qy_v=\displaystyle \frac{1}{\sqrt{\varphi}}(q p_u-p q_u).$$
Therefore, $${\textbf z} (u,v):=
(x(u,v),y(u,v),z(u,v),p(u,v),q(u,v)):\Gamma_R\flecha \R^5$$ is a
solution to system

  \begin{equation}\label{sfacil2p}
  \left(\begin{array}
  {l}x\\y\\z\\p\\q
  \end{array}\right)_v=M\left(\begin{array}
  {l}x\\y\\z\\p\\q
  \end{array}\right)_u,\qquad M=
  \frac{1}{\sqrt{\varphi}}\left(\begin{array}
  {ccccc}0&0&0&0&-1\\
  0&0&0&1&0\\
  0&0&0&q&-p\\
  0&-\varphi&0&0&0\\
  \varphi&0&0&0&0
  \end{array}\right).
  \end{equation}
The following claim provides a boundary regularity result for
${\textbf z}(u,v)$:

\begin{claim}\label{extanalitica}
In the above conditions, if $\varphi \in C^k(\cU)$ (resp. $\varphi
\in C^{\omega} (\cU)$), then ${\textbf z}(u,v)$ extends as a
$C^{k,\alpha}$ map $\forall \alpha\in(0,1)$ (resp. as a real
analytic map) to $\Gamma_R\cup \R$.
\end{claim}
\begin{proof}[Proof of the Claim]
 The first part of the proof follows a bootstrapping method.
Consider an arbitrary point of $\R$, which we will suppose without
loss of generality to be the origin. Also, consider for $0<\delta<R$
the domain $\D^+=\{(u,v): 0<u^2+v^2<\delta^2\}\cap \Gamma_R$.

From  \eqref{d1p} it follows that (cf. \cite{HeB})
\beq\label{laplacianos}\begin{array}{lcl}\Delta x&=&h_1(x_u^2+x_v^2)+h_2(x_uy_u+x_vy_v)+h_3(x_uy_v-x_vy_u)\\
\Delta
y&=&h_1(x_uy_u+x_vy_v)+h_2(y_u^2+y_v^2)+h_4(x_uy_v-x_vy_u)\end{array}\eeq
where the coefficients   $h_1=h_1(x,y,z,p,q),\ldots
,h_4=h_4(x,y,z,p,q)$ are
\begin{equation}\label{coefhi} \begin{array}{lcl}h_1&=& -\frac{1}{2\varphi}(\varphi_x+\varphi_z p),\\
h_2&=& -\frac{1}{2\varphi}(\varphi_y+\varphi_zq),\\
h_3&=&\frac{1}{\sqrt{\varphi
}}(-\frac{1}{2}\varphi_p),\\
h_4&=&\frac{1}{\sqrt{\varphi}}( -\frac{1}{2}\varphi_q),
\end{array} \end{equation}
all of them evaluated at ${\textbf z}(u,v)$.  On the other hand,
observe that the inequalities $$ (x_u -y_v)^2+(x_v+y_v)^2\geq 0,
\hspace{1cm} (x_u -y_u)^2+(x_v-y_v)^2\geq 0 $$ lead, respectively,
to $x_uy_v-x_v y_u\leq \frac{1}{2}(|\nabla x|^2+|\nabla y|^2)$ and
$x_uy_u+x_v y_v\leq \frac{1}{2}(|\nabla x|^2+|\nabla y|^2)$.

Hence, if we denote $Y=(x,y):\D^+\longrightarrow\Omega$, formula
\eqref{laplacianos} and the fact that $h_1,\ldots ,h_4$ are bounded
(since $\gamma\subset \cH$) yield \beq\label{ineqlaplac}|\Delta
Y|\leq c(|\nabla x|^2+|\nabla y|^2) \eeq for a certain constant
$c>0$.

Observe that $Y\in C^2(\D^+)\cap C^0(\overline{\D^+})$ with
$Y(u,0)=(0,0)$ for all $u$. Hence, we can apply   Heinz's Theorem in
\cite{He} to deduce that $Y\in
C^{1,\alpha}(\overline{\D^+_{\varepsilon}})$  for all
$\alpha\in(0,1)$, where $\D^+_{\varepsilon}=\D^+\cap
B(0,\varepsilon)  $ for a certain $0<\varepsilon<\delta$.

Now, the right hand side  terms in \eqref{d1p} are bounded in
$\overline{\D^+_{\varepsilon}}$ and so $p,q\in
W^{1,\infty}(\overline{\D^+_{\varepsilon}})$. Hence $p,q\in
C^{0,1}(\overline{\D^+_{\varepsilon}})$ (cf. \cite[pag. 154]{GiTr}).

Taking into account \beq\label{zfacil0}
 z_u=px_u+qy_u, \qquad
 z_v=px_v+qy_v,
 \eeq
we obtain $z\in C^{1,\alpha}(\overline{\D^+_{\varepsilon}})$
$\forall \alpha\in (0,1)$. Then, the right hand side functions in
\eqref{d1p} are Hölder continuous of any order in
$\overline{\D^+_{\varepsilon}}$. That is, $p,q\in
C^{1,\alpha}(\overline{\D^+_{\varepsilon}})$ $\forall
\alpha\in(0,1)$.


With this, we have from \eqref{laplacianos}  that $\Delta Y$ is
Hölder continuous in $\overline{\D^+_{\varepsilon}}$. Then, a
standard potential analysis argument (cf. \cite[Lemma 4.10]{GiTr})
ensures that $x,y\in C^{2,\alpha}(\overline{\D^+_{\varepsilon/2}})$.
Again, by formula \eqref{d1p} we have that $p,q\in
C^{2,\alpha}(\overline{\D^+_{\varepsilon/2}})$ and so, from
\eqref{zfacil0} that $z\in
C^{2,\alpha}(\overline{\D^+_{\varepsilon/2}})$.

At this point we may apply the same argument to $Y_u$ and  $Y_v$, in
order to obtain that $x,y,z,p,q\in
C^{3,\alpha}(\overline{\D^+_{\varepsilon/4}})$. A recursive process
leads to the fact that ${\textbf z}=(x,y,z,p,q)$ is $C^{k,\alpha}$
$\forall \alpha\in(0,1)$ (resp. $C^{\8}$) at the origin. As we can
do the same argument for all points of $\R$ and not just the origin,
we conclude that ${\textbf z}(u,v)\in C^{k,\alpha} (\Gamma_R\cup
\R)$ (resp.  ${\textbf z}(u,v)\in C^{\8} (\Gamma_R\cup \R)$).

Finally, suppose that $\varphi\in C^{\omega} (\cU)$. From
\eqref{d1p}, \beq\label{laplacianos2}
\def\arraystretch{1}\begin{array}{lll}
\Delta p&=&(\sqrt{\varphi\circ {\textbf z}})_uy_v-(\sqrt{\varphi\circ {\textbf z}})_vy_u \\

\Delta q&=&-(\sqrt{\varphi\circ {\textbf
z}})_ux_v+(\sqrt{\varphi\circ {\textbf
z}})_vx_u \\
\Delta z&=&p_ux_u+p_vx_v+q_uy_u+q_vy_v+p\Delta x+q \Delta y.\\

\end{array}\eeq
Therefore, $\textbf{z} (u,v)$ satisfies \beq\label{mixto} \Delta
\textbf{z} = h(\textbf{z},\textbf{z}_u,\textbf{z}_v)\eeq where
$h:\mathcal{O}\subset \R^{15}\flecha \R^5$ is a real analytic
function on an open set $\mathcal{O}$ of $\R^{15}$ containing the
closure of the bounded set $\{({\textbf z},{\textbf z}_u,{\textbf
z}_v)(u,v): (u,v)\in \Gamma_R\}$. Moreover, if we write
$$\textbf{z} (u,v)=(\psi(u,v),\phi(u,v)):\Gamma_R\flecha \R^3\times \R^2\equiv \R^5$$ where $\psi(u,v)$ is given by
\eqref{grafconf} and $\phi(u,v)=(p(u,v),q(u,v))$, then we see that
${\textbf z}(u,v)$ is a solution to \eqref{mixto} that meets the
mixed initial conditions
$$\left\{ \def\arraystretch{1.5} \begin{array}{l} \psi(u,0)=(0,0,0), \\ \phi_v(u,0)= (0,0).\end{array}\right.$$
As ${\textbf z}\in C^{\8}(\Gamma_R\cup \R)$ by the previous
bootstrapping argument, we are in the conditions to apply  Theorem 3
in \cite{Mu} to $\textbf{z}$ around every point in $\R$. Thus, we
deduce that $\textbf{z}$ is real analytic in $\Gamma_R\cup\R$, which
concludes the proof of Claim \ref{extanalitica}. \end{proof}

It follows from Claim \ref{extanalitica} that the functions $p(u,v)$
and $q(u,v)$ extend $C^{k,\alpha}$-smoothly $\forall\alpha\in(0,1)$
(resp. analytically) to $\Gamma_R\cup \R$, so that
$(\alfa(u),\beta(u)):=(p(u,0),q(u,0))$ is a $2\pi$-periodic map. Let
now $\gamma\subset \R^2$ denote the limit gradient of $z(x,y)$.
Then, clearly $\gamma=\{(\alfa(u),\beta(u)):u\in \R\}$, and so we
get that $\gamma$ is a closed curve in $\R^2$, possibly with
singularities, that can be parameterized as a $2\pi$-periodic
$C^{k,\alpha}$ (resp. analytic) function as
$\gamma(u)=(\alfa(u),\beta(u))$ in terms of the conformal parameters
$(u,v)$ associated to the solution $z(x,y)$.

\begin{claim}\label{gammareg}
$\gamma'(u)\neq (0,0)$ for every $u\in \R$.
\end{claim}
\begin{proof}
We start by proving that $\gamma'(u)$ can only vanish at most two
points in $[0,2\pi)$. Indeed, assume arguing by contradiction that
$\gamma'(u_1)=\gamma'(u_2)=\gamma'(u_3)=0$ for three distinct values
$u_1,u_2,u_3\in [0,2\pi)$. Since $x(u,0)=0$ for every $u\in \R$, by
\eqref{d1p} we see that $Dx (u_i,0)=0$, $i=1,2,3$. Noting then that
$x(u,v)$ satisfies the elliptic PDE \eqref{laplacianos}, and that
the zero function is another solution to the same PDE, we deduce by
Theorem $\dag$ in \cite{HaWi} that $x_{uu} x_{vv}-x_{uv}^2 <0$ in a
punctured neighborhood of each $(u_i,0)$ in the $u,v$-plane. In
other words, the axis $v=0$ is a nodal curve of $x(u,v)$ that is
crossed at $u_1,u_2,u_3$ by three other nodal curves
$\delta_1,\delta_2,\delta_3$ at a positive angle.

Next, observe that the map \eqref{grafconf} is a diffeomorphism from
$\Gamma_R :=\Sigma_R /(2\pi \Z)$ into the graph $G=\{(x,y,z(x,y)):
(x,y)\in \Omega\}$. As $G$ is a graph, $G\cap \{x=0\}\subset \R^3$
is formed by exactly two regular curves with an endpoint at the
origin. Thus, there cannot exist three nodal curves of $x(u,v)$ in
$\Gamma_R$. This contradiction shows that $\gamma'(u)$ vanishes at
most at two points in $[0,2\pi)$.

As a consequence, as $\gamma(\R)$ is convex, we can choose
$u_1,u_2\in [0,2\pi)$ with $\gamma'(u_i)\neq (0,0)$ for $i=1,2$, and
such that the respective support lines to $\gamma$ passing through
$\gamma(u_1)$ and $\gamma(u_2)$ are both tangent to a certain
direction $v_{\theta}= (\cos \theta, \sin \theta)\in \S^1$. In
particular, $-\sin \theta \alfa'(u_i) + \cos \theta \beta'(u_i)=0$
holds for $i=1,2$. Using \eqref{d1p} and the fact that
$x(u,0)=y(u,0)=0$ for every $u\in \R$, we deduce that $x_{\theta}
(u,v):= \cos \theta x(u,v) +\sin \theta y(u,v)$ satisfies
$x_{\theta} (u,0)=0$ for every $u\in \R$, and $Dx_{\theta} (u_1,0)=D
x_{\theta} (u_2,0)=(0,0)$. Also, a computation using
\eqref{laplacianos} shows that, if we denote $y_{\theta} (u,v):=
-\sin \theta x(u,v) + \cos \theta y(u,v)$, then $x_{\theta} (u,v)$
satisfies the elliptic PDE
\begin{equation}
   \begin{array}{lcl}
  \Delta x_{\theta}&=& H_1 ((x_{\theta})_u^2+(x_{\theta})_v^2)+H_2 ((x_{\theta})_u (y_{\theta})_u
  +(x_{\theta})_v(y_{\theta})_v) \\&& \\&&+H_3 ((x_{\theta})_u (y_{\theta})_v- (x_{\theta})_v (y_{\theta})_u )
 \end{array}\end{equation}
where the coefficients $H_i:=H_i(u,v)\in C^{k-1}(\Gamma_R\cup \R)$
are given in terms of the functions in \eqref{coefhi} by
$$\begin{array}{lcl}
H_1&=& (h_1\circ {\bf z})\cos \theta +(h_2\circ {\bf z})\sin \theta,\\
H_2&=& (h_2\circ {\bf z})\cos \theta -(h_1\circ {\bf z})\sin \theta,\\
H_3&=& (h_3\circ {\bf z})\cos \theta +(h_4\circ {\bf z})\sin \theta.
\end{array} $$  As the zero function is
also a solution of this PDE, we can deduce again by Theorem $\dag$
in \cite{HaWi} that $x_{\theta}(u,v)$ has two nodal curves
$\gamma_1,\gamma_2$ that intersect at a positive angle the nodal
curve $v=0$ at the points $(u_1,0)$, $(u_2,0)$. Geometrically, the
restriction of these two nodal curves $\gamma_1,\gamma_2$ to
$\Gamma_R$ corresponds (as explained above for the case $\theta=0$)
to the intersection of the graph $G$ with the plane $\cos \theta x +
\sin \theta y=0$ in $\R^3$. In particular, the axis $v=0$ cannot be
crossed by any other nodal curve of $x_{\theta} (u,v)$.

Finally, note that if $\gamma'(\xi)=(0,0)$ for some $\xi\in
[0,2\pi)$, then by \eqref{d1p} we would have $x_{\theta} (\xi)=0$
and $D x_{\theta} (\xi)=(0,0)$. Therefore, there would exist a nodal
curve of $x_{\theta}(u,v)$ crossing the $v=0$ axis at $\xi$. Thus,
$\xi=u_1$ or $\xi=u_2$, which is a contradiction since we initially
chose $u_1,u_2$ to be regular points of $\gamma$. Thus,
$\gamma'(u)\neq (0,0)$ for every $u\in \R$, which proves Claim
\ref{gammareg}.
\end{proof}

\begin{claim}\label{gammareg2}
The regular curve $\gamma(u)$ is strictly locally convex and
negatively oriented, i.e. it holds $\alfa''(u)
\beta'(u)-\alfa'(u)\beta''(u)>0$ for every $u\in \R$.
\end{claim}
\begin{proof}
Let $\psi:\Gamma_R\flecha \R^3$ be the conformal parametrization of
the graph $z=z(x,y)$ given in \eqref{grafconf}. So, $\psi$ is an
immersion with unit normal
 \begin{equation}\label{unitnormal}
 N(u,v)=\frac{(-p(u,v),-q(u,v),1)}{\sqrt{1+p(u,v)^2+q(u,v)^2}}
:\Gamma_R \flecha \S^2.
 \end{equation}
By Claim \ref{extanalitica}, $\psi,N \in C^{k,\alfa} (\Gamma_R\cup
\R)$, with
$N(u,0)=(-\alfa(u),-\beta(u),1)/\sqrt{1+\alfa(u)^2+\beta(u)^2}$ and
$\psi(u,0)=(0,0,0)$. In particular, it follows from Claim
\ref{gammareg} that $N(u,0)$ is a $2\pi$-periodic regular curve in
$\S^2$. Moreover, a simple computation shows that $N(u,0)$ has
negative geodesic curvature in $\S^2$ at every point if and only if
$\alfa''(u) \beta'(u) - \alfa' (u)\beta''(u) > 0$ for every $u$.

Note that the metric $ds^2$ in \eqref{metricapura} is conformally
equivalent to the second fundamental form of the graph $z=z(x,y)$.
Thus, if we write $w=u+iv$, the first and second fundamental forms
of $z=z(x,y)$ with respect to this parametrization are written as
\begin{equation}\label{fundamx}
\left\{\def\arraystretch{1.5} \begin{array}{rrc} I=\esiz d\psi,d\psi
\esde& =& Q \, dw^2 + 2\mu |dw|^2 +
\bar{Q} d\bar{w}^2, \\
II=-\esiz d\psi,dN\esde&=& 2\rho |dw|^2,
\end{array}\right.
\end{equation}
where $Q:= \esiz \psi_w,\psi_w\esde:\Gamma_R\cup \R\flecha \C$,
(recall that $\parc_w:=(\parc_u-i\parc_v)/2$), and $\mu,\rho
:\Gamma_R\cup \R \flecha (0,\8)$ are positive real functions. By
Claim \ref{extanalitica}, $Q,\mu,\rho$ are $C^{k-1,\alfa}$-smooth in
$\Gamma_R\cup \R$.

Also, note that by \eqref{purema} the Gaussian curvature $K$ of
$z=z(x,y)$ is $$K=\frac{z_{xx} z_{yy}-z_{xy}^2}{(1+z_x^2+z_y^2)^2}=
\frac{\varphi(x,y,z,z_x,z_y)}{(1+z_x^2+z_y^2)^2},$$ that is, in
terms of the conformal parameters $(u,v)$ we have
$$K(u,v)=\frac{\varphi\circ {\bf z}
(u,v)}{(1+p(u,v)^2+q(u,v)^2)^2}\in C^{k}(\Gamma_R\cup \R),$$ and so
$K(u,v)>0$ for all $(u,v)\in \Gamma_R\cup \R$ since $\gamma\subset
\cH$. A direct computation using \eqref{unitnormal} shows that, in
$\Gamma_R\cup \R$,

\begin{equation}\label{provex}
N\times N_u = -\sqrt{K} \psi_v, \hspace{1cm} N\times N_v = \sqrt{K}
\psi_u,
\end{equation}
where $\times$ denotes the cross product in $\R^3$. From here,
\begin{equation}\label{forQ} \def\arraystretch{2}\begin{array}{lll}
Q(u,0)&=&\displaystyle\frac{1}{4}\left(\esiz \psi_u,\psi_u\esde -
\esiz \psi_v,\psi_v \esde -2 i \esiz \psi_u,\psi_v\esde\right)(u,0)
\\ & =& -\displaystyle\frac{1}{4}\esiz\psi_v,\psi_v\esde (u,0)=
\displaystyle\frac{-1}{4K}\esiz N\times N_u,N\times N_u\esde (u,0)
\\ & = & \displaystyle\frac{-1}{4K} \esiz N_u,N_u\esde (u,0).
\end{array}\end{equation}
In particular, since $\esiz N_u,N_u\esde (u,0) >0$ for every $u$ as
we explained above, we may assume by choosing a smaller $R>0$ if
necessary that $Q$ does not vanish on $\Gamma_R\cup \R$. Using now
that $K={\rm det} (II)/{\rm det (I)}$ on $\Gamma_R$ and the previous
boundary regularity we get from \eqref{fundamx} that
\begin{equation}\label{roca}
\rho^2=K(\mu^2-|Q|^2) \hspace{1cm} \text{ in } \Gamma_R\cup \R.
\end{equation}
Since $Q\neq 0$, \eqref{roca} implies the existence of a function
$\omega\in C^{k-1,\alfa}(\Gamma_R\cup \R)$ such that $\mu = |Q|\cosh
\omega$ and $\rho=\sqrt{K} |Q| \sinh \omega$. Note that $\omega>0$
on $\Gamma_R$ and $\omega(u,0)=0$ for every $u\in \R$. In
particular, we can rewrite \eqref{fundamx} as
\begin{equation}\label{fundamxx}
\left\{\def\arraystretch{1.5} \begin{array}{rrc} I=\esiz d\psi,d\psi
\esde& =& Q \, dw^2 + 2|Q|\cosh \omega |dw|^2 +
\bar{Q} d\bar{w}^2, \\
II=-\esiz d\psi,dN\esde&=& 2\sqrt{K}|Q|\sinh\omega |dw|^2.
\end{array}\right.
\end{equation}
A standard derivation of the Gauss-Codazzi equations for $\psi$ in
terms of the complex parameter $w=u+iv$ and the data $K,Q,\omega$
(see for instance \cite{Bob}, pp. 118-119) shows that the function
$\omega$ satisfies

\begin{equation}\label{pdebob}
\omega_{w\bar{w}} + U_{\bar{w}} -V_{w} + K |Q| \sinh \omega =0,
\end{equation}
where $$U=\frac{-K_{\bar{w}} Q}{4K |Q|} \sinh \omega, \hspace{1cm}
V=\frac{K_{w} \overline{Q}}{4K |Q|} \sinh \omega.$$ In terms of the
parameters $(u,v)$, \eqref{pdebob} is a quasilinear elliptic PDE for
$\omega$ of the type
 \begin{equation}\label{laplaom}
\Delta \omega +a_1\, \omega_u \cosh \omega +  a_2 \, \omega_v \cosh
\omega + a_3\sinh \omega =0,
\end{equation}
where $a_i=a_i(u,v)\in C^{k-2}(\Gamma_R\cup \R)$. Observe that
$\omega=0$ is a trivial solution to \eqref{laplaom}.

Moreover, if we denote $\sigma(u):=N(u,0):\R/(2\pi \Z)\flecha \S^2$,
we have using \eqref{provex}, \eqref{forQ}

$$\def\arraystretch{1.5}\begin{array}{lll}
\esiz \sigma'',\sigma\times \sigma'\esde & = & \esiz N_{uu}, N\times
N_u\esde (u,0)= \sqrt{K} \esiz N\times \psi_{uv},N\times N_u\esde
(u,0) \\ & =& \sqrt{K} \esiz \psi_{uv},N_u\esde (u,0)= \sqrt{K}
\left( \frac{\parc}{\parc v} (\esiz \psi_u,N_u\esde ) - \esiz
\psi_u,N_{uv}\esde \right) (u,0) \\ &=& \sqrt{K} \frac{\parc}{\parc
v} (\esiz \psi_u,N_u\esde) (u,0) = - 2 K |Q|  \omega_v \cosh \omega
(u,0)= -2 K |Q|  \omega_v (u,0) \\ &=& -\frac{1}{2} \esiz
\sigma',\sigma'\esde \omega_v (u,0).
\end{array}$$
Therefore, $$\omega_v(u,0) = -\frac{2\esiz
\sigma''(u),\sigma(u)\times \sigma'(u)\esde}{\esiz
\sigma'(u),\sigma'(u)\esde}= -2 ||\sigma'(u)|| \kappa_{\sigma}
(u),$$ where $\kappa_{\sigma}$ denotes the geodesic curvature of
$\sigma$ in $\S^2$. Let us recall at this point that the real axis
is a nodal curve of $\omega$. Since $\omega$ is a solution to the
elliptic PDE \eqref{laplaom}, by Theorem $\dag$ in \cite{HaWi} we
deduce that, at the points $(u,0)$ where $\omega_v (u,0)= 0$ there
exists at least one nodal curve of $\omega$ that crosses the real
axis at a definite angle. But this situation is impossible, since
$\omega>0$ in $\Gamma_R$. Therefore we see that $\omega_v (u,0)> 0$
for every $u$. Consequently, the geodesic curvature of $N(u,0)$ in
$\S^2$ is strictly negative. As explained previously, this condition
implies that $\alfa''(u) \beta'(u) - \alfa'(u)\beta'' (u) >0$ for
every $u\in \R$. This proves Claim \ref{gammareg2}.
\end{proof}

We observe that these three claims together with the paragraph above
Claim \ref{gammareg} prove the first two items in Theorem
\ref{mainth22}.

In order to prove item (3) of Theorem \ref{mainth22}, assume that
$\varphi$ (and hence any solution to \eqref{purema}) is analytic.
Observe that the map ${\bf z}(u,v)$ can be recovered in terms of an
analytic, $2\pi$-periodic curve $\gamma(u)=(\alfa(u),\beta(u))$ as
the unique solution to the Cauchy problem for the system
\eqref{sfacil2p} with the initial condition
 \begin{equation}\label{inconfac}
 \textbf{ z}(u,0)=(0,0,0,\alfa(u),\beta(u)).
 \end{equation}
Also, observe that the parameters $(u,v)\in\Gamma_R$ associated to
the solution $z$ of \eqref{purema} are defined up to $2\pi$-periodic
conformal changes of $\Gamma_R$ that simply yield regular, analytic
reparametrizations
 of the limit gradient $\gamma$.

Taking this into account, we deduce by the uniqueness of the
solution to the Cauchy problem for system \eqref{sfacil2p} that if
two solutions $z,z'$ to \eqref{purema} with $\varphi\in C^{\omega}
(\cU)$ satisfy: (i) $z_{xx}>0$, $z'_{xx}>0$, (ii) $z$ and $z'$ have
an isolated singularity at the origin, and (iii) both $z$ and $z'$
have the same limit gradient $\gamma\subset \cH$ at the singularity,
then their graphs agree on a neighborhood of the origin. This
finishes the proof of Theorem \ref{mainth22}.
\end{proof}

\section{Existence of isolated singularities and the proof of Theorem \ref{main}}\label{sec: existencia}

In this section we consider the general elliptic equation of
Monge-Ampère type in dimension two, i.e. the fully nonlinear PDE
\eqref{eq0}. Note that \eqref{eq0} can be rewritten as
 \beq\label{eq1}
   Az_{xx}+2B z_{xy}+C z_{yy} +z_{xx} z_{yy}-z_{xy}^2=E,\eeq
where $A=A(x,y,z,z_x,z_y),\dots, E=E(x,y,z,z_x,z_y)$ are defined on
an open set $\cU\subset \R^5$ and satisfy on $\cU$ the ellipticity
condition \beq\label{elliptic} \mathcal{D}:=AC-B^2+E>0. \eeq

We can also rewrite \eqref{eq1} as
$(A+z_{yy})(C+z_{xx})-(B-z_{xy})^2= \mathcal{D}>0$, from where we
see that $C+z_{xx}$ is never zero.

The next theorem provides a general existence result for solutions
to \eqref{eq0} with an isolated singularity at the origin and a
prescribed limit gradient at the singularity. We recall the
definition of $\Gamma_R$, $\Sigma_R$ in Section 2, and denote
$\widehat{\Gamma_R}:=\widehat{\Sigma_R}/(2\pi \Z)$ where
$\widehat{\Sigma_R}=\{w\in \C: -R<{\rm Im} (w) <R\}$.

\begin{teorema}\label{mainth11}
Assume that the coefficients $A,\dots, E$ are real analytic in
$\cU$. Let $\gamma(u)=(\alfa(u),\beta(u))$ be a real analytic,
$2\pi$-periodic curve such that $(0,0,0,\gamma(\R))\subset \cU$.

Then, there exists a real analytic map $\psi
:\widehat{\Gamma_R}\flecha \R^3$ such that:
 \begin{enumerate}
   \item
$\psi(u,0)=(0,0,0)$ for every $u\in \R$.
 \item
There exists a real analytic map $(p,q):\widehat{\Gamma_R}\flecha
\R^2$ such that $(p,q)(u,0)=\gamma(u)$ for every $u\in\R$ and
$(\psi,p,q)(\Gamma_R)\subset \cU$. Moreover, the map
$N(u,v):\widehat{\Gamma_R}\flecha \S^2$  defined by
$$N(u,v)=\frac{(-p,-q,1)}{\sqrt{1+p^2+q^2}}(u,v)$$
satisfies that $\esiz \psi_u,N\esde =\esiz \psi_v,N\esde =0$ in
$\Gamma_R$.
 \item
Assume that the map $(x(u,v),y(u,v))$ is an orientation preserving
local diffeomorphism at some point $(u_0,v_0)\in \Gamma_R$. Then,
the image of $\psi$ around that point is the graph $G\subset \R^3$
of some real analytic solution $z=z(x,y)$ to \eqref{eq1} for the
coefficients $A,\ldots ,E$ such that $C+z_{xx}>0$.
 \item
If $\gamma(u)$ is a regular, negatively oriented, strictly convex
parametrized Jordan curve (so, both $-||\gamma'(u)||$ and the
curvature of $\gamma(u)$ are strictly negative for every $u$), then
for $R>0$ small enough, $\psi(\Gamma_R)$ is the graph of a solution
$z$ to \eqref{eq1} for the coefficients $A,\dots, E$ which is
defined on a punctured neighborhood around the origin, and has an
isolated singularity at the puncture. Moreover, the limit gradient
of this solution is the curve $\gamma=\gamma(\R)$, and $C+z_{xx}>0$.

\end{enumerate}
\end{teorema}
\begin{observacion}
The first three items of Theorem \ref{mainth11} prove that, if we
start from a $2\pi$-periodic, real analytic curve $\gamma(u)$ in
$\R^2$, we can construct from $\gamma(u)$ a multivalued solution to
\eqref{eq1} with a singularity at the origin. Here by a
\emph{multivalued solution} we mean a surface such that whenever it
is transverse to the vertical direction around one point, it is a
local solution to \eqref{eq1} around this point. If $\gamma(u)$ is
regular and strictly locally convex but non-embedded, the
singularity at the origin of the corresponding multivalued solution
is isolated. See \cite{CaLi} for a different study of multivalued
solutions to Monge-Ampère equations.
\end{observacion}
We also observe that Theorem \ref{mainth11} implies the following
corollary.
\begin{corolario}\label{cormain11}
Assume that the coefficients $A,\dots, E$ are real analytic in
$\cU\subset \R^5$. Let $\gamma\subset \R^2$ be a real analytic,
regular, strictly convex Jordan curve such that
$(0,0,0,\gamma(\R))\subset \cU$. Then there exists a solution $z$ to
\eqref{eq1} for these coefficients that has an isolated singularity
at the origin, and whose limit gradient at the singularity is
$\gamma$.

\end{corolario}

Before proving Theorem \ref{mainth11}, let us make some comments
about solutions to the general equation \eqref{eq1}. Let $z$ be a
solution to \eqref{eq1} on some domain $W\subset\R^2$, where
$A,\dots, E$ are of class $C^2$, so that $z$ is of class
$C^{3,\alfa}$ on compact sets of $W$. By the ellipticity condition
\eqref{elliptic}, \beq\label{metrica}
ds^2=(z_{xx}+C)dx^2+2(z_{xy}-B)dxdy+(z_{yy}+A)dy^2 \eeq is a
Riemannian metric on $W$. Then, $(W,ds^2)$ admits in a neighborhood
of each point of $W$ conformal parameters $w:=u+iv$ of class $C^2$
such that (cf. \cite{HeB})

\beq\label{conforme}ds^2=\frac{\sqrt{\mathcal{D}}}{u_xv_y-u_yv_x}|dw|^2.
\eeq where $(x,y)$ satisfy $x_uy_v-x_vy_u>0$. From \cite{Bey1} we
have the equations,
 \beq \label{d1}\begin{array}{l}
  p_u=\sqrt{ \mathcal{D} }y_v+ B y_u- C x_u,\\
  p_v=-\sqrt{\mathcal{D}}y_u+ B y_v- C x_v,\\
  q_u=-\sqrt{ \mathcal{D} }x_v+ B x_u- A y_u,\\
  q _v=\sqrt{ \mathcal{D} }x_u+ B x_v- A y_v,\\
  \end{array} \eeq
from where, since

  $$\def\arraystretch{1.5}\begin{array}
 {ll}z_v&=px_v+qy_v\\
 &=-\displaystyle\frac{p}{\sqrt{\mathcal{D}}}(q_u-Bx_u+Ay_u)+\frac{q}{\sqrt{\mathcal{D}}}(p_u-By_u+Cx_u)\\
 &=\displaystyle \frac{1}{\sqrt{\mathcal{D}}}(x_u(B p+C q)-y_u(A p+B q)+q p_u-p
 q_u),
 \end{array}$$
we arrive at the the following system which generalizes
\eqref{sfacil2p}:

  \begin{equation}\label{sfacil2}
  \left(\begin{array}
  {l}x\\y\\z\\p\\q
  \end{array}\right)_v=\tilde{M}\left(\begin{array}
  {l}x\\y\\z\\p\\q
  \end{array}\right)_u,\qquad \tilde{M}=
  \frac{1}{\sqrt{\mathcal{D}}}\left(\begin{array}
  {ccccc}B&-A&0&0&-1\\
  C&-B&0&1&0\\
  Bp+Cq&-Ap-B q&0&q&-p\\
  0&-E&0&B&C\\
  E&0&0&-A&-B
  \end{array}\right).
  \end{equation}

\vspace{0.2cm}

\emph{Proof of Theorem \ref{mainth11}:}

Let $\gamma(u)=(\alfa(u),\beta(u))$ be a real analytic,
$2\pi$-periodic curve in $\R^2$, and assume that $A,\dots, E$ are
real analytic functions on an open set $\cU\subset \R^5$ that
contains $(0,0,0,\gamma(\R))$, and that satisfy the ellipticity
condition \eqref{elliptic}. Let us consider the  $2\pi$-periodic
initial data $(0,0,0,\alpha(u),\beta(u))$ along the axis $v=0$ in
the $(u,v)$-plane for the system \eqref{sfacil2}. By the
Cauchy-Kowalevsky theorem, there exists a unique real analytic
solution $(x,y,z,p,q)$ to \eqref{sfacil2}, defined on a neighborhood
$\hat{\Sigma_R}=\{(u,v): -R<v<R\}$ of the axis $v=0$, such that
\begin{equation}\label{incondi}(x,y,z,p,q)(u,0)= (0,0,0,\alpha(u),\beta(u)).\end{equation} Observe that
$\Psi:=(x,y,z,p,q): \hat{\Sigma_R}\flecha \R^5$ is $2\pi$-periodic
with respect to $u$, i.e. it is well defined on the quotient
$\hat{\Gamma_R}:=\hat{\Sigma_R}/(2\pi \Z)$.

A computation from \eqref{sfacil2} proves the relation

\beq\label{integfacil}p_vx_u+q_vy_u=p_ux_v+q_uy_v,\eeq which is the
integrability condition needed for the existence of some smooth
function $z_0$ on $\widehat{\Sigma_R}$, unique up to an additive
constant, such that
$$(z_0)_u= p x_u + q y_u,\hspace{1cm} (z_0)_v= p x_v+q y_v.$$
If follows from \eqref{sfacil2} that $(z_0)_v=z_v$ and so
$z(u,v)=z_0(u,v)+f(u)$ for some $f\in C^{\omega} (\Gamma_R\cup \R)$.
Also, observe that \eqref{incondi} implies that $z(u,0)\equiv 0$ and
$(z_0)_u (u,0)\equiv 0$. Thus, $f(u)$ must be constant, and as $z_0$
was defined up to additive constants we may assume that
$z(u,v)=z_0(u,v)$. In particular, it holds \beq\label{zfacil}
 z_u=px_u+qy_u, \qquad
 z_v=px_v+qy_v.
 \eeq
Defining now
\begin{equation}\label{mapasi}
\psi(u,v):=(x(u,v),y(u,v),z(u,v)):\hat{\Gamma_R}\flecha \R^3
 \end{equation}
and
 \begin{equation}\label{mapan}
N(u,v):=\frac{(-p(u,v),-q(u,v),1)}{\sqrt{1+p(u,v)^2
+q(u,v)^2}}:\hat{\Gamma_R}\flecha \S^2
 \end{equation} we see that the first two
items of Theorem \ref{mainth11} hold.

To prove item 3, suppose now that the map $(x(u,v),y(u,v))$ is an
orientation preserving local diffeomorphism at some point
$(u_0,v_0)\in \Sigma_R$, i.e. the condition

\beq\label{det}J:=x_uy_v-x_vy_u>0 \eeq holds at this point. Thus,
around $(u_0,v_0)$ the image of the map $\psi(u,v)$ is the graph $G$
in $\R^3$ of a real analytic function $z=z(x,y)$, and from formula
\eqref{zfacil} the relations $z_x=p$ and $z_y=q$ hold. We prove next
that $z(x,y)$ is a solution to \eqref{eq1} for the coefficients
$A,\dots, E$ we started with.

If we denote $r=z_{xx}$, $s=z_{xy}$, $t=z_{yy}$, then using
\eqref{d1} and working in terms of the $(u,v)$ coordinates we obtain
\begin{equation}\label{fin11}\sqrt{\mathcal{D}} y_v = p_u - B y_u + C x_u = (C+r) x_u -(B-s)
y_u,\end{equation} and working similarly,
\begin{equation}\label{fin12}\def\arraystretch{1.3}\begin{array}{lll} \sqrt{\mathcal{D}} y_u &=&
-(C+r) x_v +
(B-s) y_v, \\ \sqrt{\mathcal{D}} x_v &=& (B-s) x_u - (A+t) y_u, \\
\sqrt{\mathcal{D}} x_u &= &-(B-s) x_v + (A+t)
y_v.\end{array}\end{equation} After the change of coordinates
$(u,v)\mapsto (x,y)$, these expressions yield

\beq\label{e5}\def\arraystretch{1.8}\begin{array}{cc}
u_x=\displaystyle\frac{(C+r)v_y+(B-s)v_x}{\sqrt{\mathcal{D}}},&
v_x=\displaystyle\frac{-(C+r)u_y-(B-s)u_x}{\sqrt{\mathcal{D}}},\\
u_y=\displaystyle\frac{-(B-s) v_y -(A+t) v_x}{\sqrt{\mathcal{D}}}, &
v_y=\displaystyle\frac{ (B-s) u_y +(A+t)
u_x}{\sqrt{\mathcal{D}}}.\end{array}\eeq We deduce then from  the
second and fourth equation in \eqref{e5}    that the system
\beq\label{m1} \left(
                                                                                \begin{array}{c}
                                                                                  v_x \\
                                                                                  v_y \\
                                                                                \end{array}
                                                                              \right)=\mathfrak{M}_1 \left(
                                                                                \begin{array}{c}
                                                                                  u_x \\
                                                                                  u_y \\
                                                                                \end{array}
                                                                              \right)
\eeq holds, where
$$\mathfrak{M}_1=\frac{1}{\sqrt{\mathcal{D}}}\left(\begin{array}{cc}-(B-s) &
-(C+r)\\A+t & B-s
 \end{array}\right).$$ Similarly, from   the first and the third equation in \eqref{e5} we get \beq \label{m2}\left(
                                                                                \begin{array}{c}
                                                                                  u_x \\
                                                                                 u_y \\
                                                                                \end{array}
                                                                              \right)=\mathfrak{M}_2 \left(
                                                                                \begin{array}{c}
                                                                                  v_x \\
                                                                                  v_y \\
                                                                                \end{array}
                                                                              \right),
 \eeq  where $$\mathfrak{M}_2=\frac{1}{\sqrt{\mathcal{D}}}\left(\begin{array}{cc}B-s &  C+r\\-(A+t) & -(B-s)
 \end{array}\right).$$ Clearly, $\mathfrak{M}_1$ is proportional to
 $\mathfrak{M}_2^{-1}$, i.e. $\mathfrak{M}_1\mathfrak{M}_2 =
 \landa(x,y) {\rm Id}$ for some function $\landa$. Hence, from \eqref{m1} and \eqref{m2} we
 obtain $\landa=1$, i.e.
$\mathfrak{M}_1\mathfrak{M}_2={\rm Id}$, and so
$$(A+t)(C+r)-(B-s)^2=\mathcal{D}.$$ That is, $z(x,y)$ is a solution to \eqref{eq1}, as we wanted to
show. Besides, a computation from \eqref{fin11} and \eqref{fin12}
shows that $$C+r= \frac{\sqrt{\mathcal{D}} (y_u^2+y_v^2)}{x_u y_v
-x_v y_u},$$ which is positive by \eqref{det}. This completes the
proof of item (3).

To prove item (4), assume that $\gamma(u)=(\alfa(u),\beta(u))$ is
also regular, embedded, negatively oriented and strictly convex,
i.e. $\alfa''(u) \beta'(u)-\beta''(u) \alfa'(u) >0$ for every $u$.
If we let $J$ be the function in \eqref{det}, then $J(u,0)=0$ for
every $u$, and a computation from \eqref{d1} at $(u,0)$ yields

 \begin{equation}\label{wderivada}
\def\arraystretch{2}\begin{array}{ll}J_v(u,0)&=(x_{uv}y_v-x_vy_{uv})(u,0)\\
&=\displaystyle\left(\left(-\frac{\beta'}{\sqrt{\mathcal{D}}} \right)_u\frac{\alpha'}{\sqrt{\mathcal{D}}}+\frac{\beta'}{\sqrt{\mathcal{D}}}\left(\frac{\alpha'}{\sqrt{\mathcal{D}}}\right)_u\right)(u,0)\\
&=\displaystyle\frac{-1}{\mathcal{D}^{3/2}}\left((\beta''\sqrt{\mathcal{D}}-\beta'(\sqrt{\mathcal{D}})_u)\alpha'-
\beta'(\alpha''\sqrt{\mathcal{D}}-\alpha'(\sqrt{\mathcal{D}})_u)\right)(u,0)\\
&=\displaystyle\frac{1}{\mathcal{D}}(\beta'(u)\alpha''(u)-\beta''(u)\alpha'(u))>0.\end{array}
 \end{equation}

Consequently, since $J$ is $2\pi$-periodic, there is some $R>0$ such
that $J>0$ on $\Gamma_R$. In particular, the map
$\psi(u,v):\hat{\Gamma_R}\flecha \R^3$ given by \eqref{mapasi}
satisfies:

\begin{enumerate}
  \item
The projection $(x(u,v),y(u,v)):\Gamma_R\flecha \R^2$ is an
orientation preserving local diffeomorphism.
 \item
$\psi(u,0)=0$ for every $u\in \R$.
 \item
The upwards-pointing unit normal $N:\Gamma_R\flecha \S_+^2$ of
$\psi$ restricted to $\Gamma_R$ extends analytically to
$\hat{\Gamma_R}$, and \eqref{mapan} holds.
 \item
For every point $(u_0,v_0)\in \Gamma_R$ there exists some $\delta>0$
such that the restriction of $\psi$ to the disk of radius $\delta$
centered at $(u_0,v_0)$ is a graph $z=z(x,y)$ which satisfies
\eqref{eq1} and $C+z_{xx}>0$.
\end{enumerate}

We need to prove now that for $R>0$ small enough, $\psi(\Gamma_R)$
is a graph of a function $z=z(x,y)$ over a punctured disc
$\Omega\subset \R^2$.

Since the set of points $(0,0,0,\gamma(u))$, $u\in\R$, is a compact
set contained in $\cU$, and $A,B,C$ are continuous functions in
$\cU$, there exist sufficiently large real constants $a,c$ such that
the inequalities
$$
a-A>0,\qquad c-C>0,\qquad (c-C)(a-A)-B^2>0,
$$
are satisfied in $\hat{\Gamma_{R'}}$, for a certain positive real
number $R'\leq R$. With no loss of generality we will assume $R=R'$.

Using item (4) above we can view $\psi(\Gamma_R)$ locally as a graph
$z=z(x,y)$ around any point $(u_0,v_0)\in \Gamma_R$, in such a way
that the expression $ds^2$ given by (\ref{metrica}) defines a
Riemannian metric around $(u_0,v_0)$. Hence, we obtain that the
matrix
$$
\left(
\begin{matrix}
r+c&s\\
s&t+a
\end{matrix}
\right)= \left(
\begin{matrix}
r+C&s-B\\
s-B&t+A
\end{matrix}
\right)+ \left(
\begin{matrix}
c-C&B\\
B&a-A
\end{matrix}
\right)
$$
is positive definite around $(u_0,v_0)$ because it is the sum of two
positive definite matrices. As $(u_0,v_0)\in \Gamma_R$ is arbitrary,
this means that the map $\psi^*:\hat{\Gamma_R}\flecha \R^3$ given by
$$\psi^*(u,v)= \left(x(u,v),y(u,v),z(u,v)+\frac{c}{2}
x(u,v)^2+\frac{a}{2} y(u,v)^2\right)$$ is a regular, strictly convex
surface in $\R^3$ when restricted to $\Gamma_R$ because
$$
z^*_{xx}z^*_{yy}-{z^*_{xy}}^2=(r+c)(t+a)-s^2>0,
$$
where $z^*=z+\frac{c}{2} x^2+\frac{a}{2} y^2$.

Also, $\psi^*(u,0)=0$ for every $u$, and the projection of
$\psi^*|_{\Gamma_R}$ into $\R^2$ is a local diffeomorphism. The unit
normal of $\psi^*$ in $\Gamma_R$ is
$$N^*(u,v)=\frac{1}{\sqrt{1+(p+cx)^2+(q+ay)^2}}(-p-cx,-q-ay,1),$$ where
$x,y,p,q$ are evaluated at $(u,v)$. We remark that
\beq\label{normal} N^*(u,0)=N(u,0)= \frac{(-\alfa(u),-\beta(u),
1)}{\sqrt{1+\alfa(u)^2+\beta(u)^2}},\eeq which is a regular,
strictly convex Jordan curve in the upper hemisphere of $\S^2$.

Consider now the analytic \emph{Legendre transform} of
$\psi^*(u,v)$, given by (see \cite[p. 89]{LSZ})
$$
\cL (u,v)= \left(-\frac{N_1^*}{N_3^*},-\frac{N_2^*}{N_3^*},-x
\frac{N_1^*}{N_3^*} -y \frac{N_2^*}{N_3^*} -
z^*\right):\hat{\Gamma_R}\flecha \R^3,
$$
where we are denoting $N^*=(N_1^*,N_2^*,N_3^*)$. It is well known
that, since $\psi^* |_{\Gamma_R}$ is a regular, locally strictly
convex surface in $\R^3$ whose projection to the $(x,y)$-plane is a
local diffeomorphism, then so is $\cL |_{\Gamma_R}$. Its
upwards-pointing unit normal is
 \begin{equation}\label{norleg}
\mathcal{N}_{\cL}= \frac{(-x,-y,1)}{\sqrt{1+x^2+y^2}}
:\Gamma_R\flecha \S_+^2,
 \end{equation}
 where $x,y$ are evaluated at $(u,v)$; hence, $\mathcal{N}_{\cL}$ can be analytically extended to $\hat{\Gamma_R}$.

Since $\cL |_{\Gamma_R}$ is locally strictly convex with $\cL(u,0)$
lying on the horizontal plane $z=0$, $\mathcal{N}_{\cL}
(u,0)=(0,0,1)$ and $z_{xx}^*=r+c>r+C>0$, we have that there exists
$R'>0$ small enough such that $\cL(\Gamma_{R'})$ lies on the upper
half-space of $\R^3$. Now, let us see that the intersection of
$\cL(\Gamma_{R'})$ with each plane $z=\varepsilon$, for
$0<\varepsilon\leq \varepsilon_0$ small enough, is a regular convex
Jordan curve in that plane.

Since $\cL(u,0)=(\alfa(u),\beta(u),0)$ is a horizontal regular
curve, we have that $\cL_u(u,0)$ is a non-vanishing tangent
horizontal vector. Thus, from the compactness of the set
$\R/(2\pi\Z)$, there exists $R''>0$ small enough such that the
horizontal projection $\pi_h(u,v)$ of the vector $\cL_u(u,v)$ does
not vanish for $0\leq v<R''$ and, in addition,  $\pi_h$ is not a
normal vector to the horizontal curve $\gamma_{\varepsilon}$  given
by the intersection of $\cL(\Gamma_{R'})$ with the plane
$z=\varepsilon$, for $0\leq\varepsilon\leq \varepsilon_0$ with
$\varepsilon_0$ small enough. In other words, $\pi_h$ is a
continuous function, $\pi_h(u,0)$ agrees with the derivative of the
strictly convex Jordan curve $(\alfa(u),\beta(u),0)$ and $\pi_h$ is
not normal to $\gamma_{\varepsilon}$. Hence, the rotation index of
$\gamma_{\varepsilon}$ is constant for all
$\varepsilon\in[0,\varepsilon_0]$, and so it must be equal to one.
In particular, the locally strictly convex curve
$\gamma_\varepsilon$ must be embedded.

Now, the piece of the surface $\cL(\Gamma_{R'})$ lying between two
of those parallel planes associated to
$0<\varepsilon_1<\varepsilon_2$ is strictly convex, and bounded by
two regular convex Jordan curves, one on each plane. In these
conditions, the unit normal of $\cL(\Gamma_{R'})$ defines a global
diffeomorphism onto some annular domain of $\S_+^2$. Letting
$\varepsilon_1\to 0$ we conclude that there exists some $R>0$ small
enough such that the unit normal \eqref{norleg} to $\cL$ restricted
to $\Gamma_R$ is a diffeomorphism onto a domain of $\S^2$. But now,
in the view of the expression \eqref{norleg}, this means that the
map $(x(u,v),y(u,v))$ restricted to this domain $\Gamma_R$ is a
global diffeomorphism onto its image. Thus, both $\psi(\Gamma_R)$
and $\psi^*(\Gamma_R)$ are graphs of functions $z(x,y)$ and
$z^*(x,y)$ over a punctured disc $\Omega\subset \R^2$.

Observe that by item (3), the function $z(x,y)$ is a solution to
\eqref{eq1} with an isolated singularity at the origin. Moreover, it
is clear from the construction process we have followed that its
limit gradient at the singularity is the curve $\gamma$ we started
with, and $C+z_{xx}>0$. This concludes the proof of item (4) and
Theorem \ref{mainth11}.

\subsection*{The proof of Theorem \ref{main}}

We are now ready to complete the proof of Theorem \ref{main}.

Assume that $\varphi$ is analytic, and consider the map sending each
$z\in \cM_1$ to the pair $(\gamma, \ep)$ given by its limit gradient
at the origin, and by $\ep=0$ if $z_{xx}>0$ and $\ep=1$ if
$z_{xx}<0$. As explained in Section \ref{sec: unicidad}, if $z$ is a
solution to \eqref{purema} with $z_{xx}<0$ and an isolated
singularity at the origin, then $\tilde{z}(x,y):=-z(-x,-y)$ is a
solution to $z_{xx} z_{yy}-z_{xy}^2=\tilde{\varphi}
(x,y,z,z_x,z_y)$, where $\tilde{\varphi}
(x,y,z,p,q)=\varphi(-x,-y,-z,p,q)$, with $\tilde{z}_{xx}>0$ and an
isolated singularity at the origin. Moreover, the limit gradients
$\gamma,\tilde{\gamma}$ of $z$ and $\tilde{z}$ at the origin
coincide, and $\cH=\tilde{\cH}$.

If $z_{xx}>0$, the fact that $\gamma\in \cM_2$ follows by item (1)
of Theorem \ref{mainth22}. If $z_{xx}<0$, the same conclusion holds
since the function $\tilde{z}$ defined above has the same limit
gradient as $z$, and satisfies $\tilde{z}_{xx}>0$. Thus, the map
$z\in \cM_1\mapsto (\gamma,\ep)\in \cM_2\times \Z_2$ is well
defined. This map is also injective by item (3) of Theorem
\ref{mainth22}, arguing again with the function $\tilde{z}$ instead
of $z$ if $z_{xx}<0$.

Finally, let us prove that this map is surjective. Let $\gamma\in
\cM_2$. By Corollary \ref{cormain11} for $A=B=C=0$, $E=\varphi$, we
see that $\gamma\in \cM_2$ is the limit gradient at the origin of
some solution $z\in \cM_1$ to \eqref{purema} for $\varphi$ such that
$z_{xx}>0$. If we apply again Corollary \ref{cormain11}, but this
time for $A=B=C=0$ and $E=\tilde{\varphi}$ with $\tilde{\varphi}$ as
above, and define in terms of the obtained solution $\tilde{z}$ a
new function $z(x,y):=-\tilde{z} (-x,-y)$, we obtain a solution to
\eqref{purema} for $\varphi$ such that $z_{xx}<0$, with an isolated
singularity at the origin and $\gamma$ as its limit gradient at the
singularity. Thus, the map $z\mapsto (\gamma,\ep)$ is surjective.
This proves Theorem \ref{main}.

\begin{nota}\label{remaruno}
Theorem \ref{main} does not hold for the general equation of
Monge-Ampère type \eqref{eq0}, as the next two examples highlight.

 \begin{enumerate}
   \item
The gradient of a solution to \eqref{eq0} can blow up at an isolated
singularity. For example, Figure \ref{figfinal} shows such behavior
on a rotational graph that satisfies a linear Weingarten relation
$aH+bK=c$, where $H,K$ are the mean and Gaussian curvature of the
graph and the constants $a,b,c$ satisfy $a^2 + bc
> 0, b\neq 0$. By the standard formulas of $H,K$ for graphs in
$\R^3$, it follows that such a rotational graph satisfies an
equation of type \eqref{eq0}.
 \item
Even when the limit gradient of a solution to \eqref{eq0} at an
isolated singularity is a convex Jordan curve, it might not be
strictly convex. For example, consider the PDE of type \eqref{eq0}

 \begin{equation}\label{pdecut}
2u_x^2 u_{xy} + u_{xx} u_{yy} -u_{xy}^2 = 1+u_x^4,
 \end{equation}
and the regular curve  $$\gamma(u)= \frac{1}{8} \left(4 \sin (2u),
4\cos(2u) +4\sin (2u) -\cos (4u)\right):\R/(2\pi \Z)\flecha \R^2.$$
The curvature of $\gamma(u)$ is always positive except at $u=0$,
where it is zero. If we now apply the construction procedure
explained in Theorem \ref{mainth11} to \eqref{pdecut} and
$\gamma(u)$, it can be checked that the corresponding function
$J(u,v)$ is positive on a strip $\{0<v<R\}$ for $R$ small enough,
and from there one can prove that it is obtained a solution to
\eqref{pdecut} with an isolated singularity at the origin whose
limit gradient is $\gamma(\R)$.
 \end{enumerate}

\end{nota}

 \begin{figure}[h]\label{figfinal}\begin{center}
 \includegraphics[width=4cm]{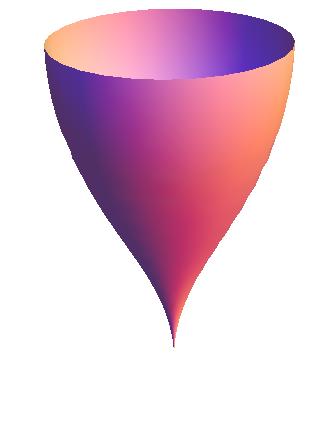}
 \caption{An isolated singularity of a rotational linear Weingarten graph, whose gradient tends to
 infinity at the puncture.} \end{center}
\end{figure}

\section*{Appendix: Isolated singularities of prescribed curvature in $\R^3$}
Let $\psi:\Omega\flecha \R^3$ be an immersion of the punctured disc
$\Omega=\{(x,y)\in\R^2 : 0<x^2+y^2<\rho^2\} $ into $\R^3$ and assume
that $\psi$ extends continuously but not $C^1$-smoothly to the
origin. Following \cite{GaMi}, we say in these conditions that
$\psi$ has an \emph{embedded isolated singularity} at
$p_0=\psi(0)\in\R^3$ if there is a punctured neighborhood $U\subset
\Omega$ of the origin such that $\psi (U)$ is an embedded surface.

In these conditions, assume moreover that $\psi:\Omega\flecha \R^3$
has positive curvature (not necessarily constant) at every point.
Then, we can orient it by choosing the unique unit normal
$N:\Omega\flecha \S^2$ with respect to which the second fundamental
form of $\psi$ is positive definite. We call this orientation the
\emph{canonical orientation} of the surface. It then follows by
Theorem 13 in \cite{GaMi} that $\psi (\Omega)$ can be viewed around
the singularity as a convex graph over a punctured disc in some
direction of $\R^3$. Specifically, there is a punctured neighborhood
$U^*\subset \Omega$ of the origin and an isometry $\Psi$ of $\R^3$
such that if $(x',y',z')=\Psi(x,y,z)$, then $\psi(U^*)$ is a convex
graph $z'=z'(x',y')$ with an isolated singularity at the origin, and
for which the unit normal $N$ associated to its canonical
orientation is $N=(-z'_{x'} \parc_{x'} - z'_{y'} \parc_{y'} +
\parc_{z'})/\sqrt{1+(z'_{x'})^2 +(z'_{y'})^2}$.

Let $\sigma\subset\S^2$ denote the \emph{limit unit normal} of
$\psi$ at the singularity, i.e. the set of points $w_0\in \S^2$ for
which there exist points $q_n\in \Omega$ converging to $(0,0)$ such
that $N(q_n)$ converge to $w_0$. It follows from the previous
discussion that $\sigma$ is explicitly related to the limit gradient
of the surface at the singularity, when we view $\psi(U^*)$ as a
graph $z'=z'(x',y')$ as explained above.

Besides, it is easy to observe that for any direction $v_0\in \S^2$
a curve $\sigma(u)$ in the hemisphere $\S^2\cap\{x\in \R^3: \esiz
x,v_0\esde>0\}$ is regular and strictly convex if and only if so is
the planar curve $\gamma(u)$ contained in the plane
$\{v_0\}^{\perp}\subset \R^3$ given by
$$\gamma (u)= \frac{\esiz \sigma(u),e_1\esde}{\esiz \sigma (u),v_0\esde}
\, e_1 +  \frac{\esiz \sigma (u),e_2\esde} {\esiz \sigma
(u),v_0\esde} \, e_2,$$ where $\{e_1,e_2,v_0\}$ is a positively
oriented orthonormal basis of $\R^3$. Also, recall that any regular
strictly convex Jordan curve in $\S^2$ is contained in some open
hemisphere of $\S^2$.

With all of this, and recalling that the equation for the curvature
$K=K(x,y)$ of a graph $z=z(x,y)$ in $\R^3$ is given by $${\rm det}
(D^2 z)=K (1+|Dz|^2)^2,$$ and is invariant by isometries of $\R^3$,
it is elementary to obtain the following theorem as a corollary of
Theorem \ref{main}.

 \begin{teorema}\label{teo: embedded}
Let $\mathcal{K}:\mathcal{O}\subset\R^3\flecha (0,\8)$ be a positive
real analytic function defined on an open  set $\mathcal{O}\subset
\R^3$ containing a given point $p_0\in\R^3$. Let $\cA_1$ denote the
class of all the canonically oriented surfaces $\Sigma$ in $\R^3$
that have $p_0$ as an embedded isolated singularity, and whose
extrinsic curvature at every point $(x,y,z)\in \Sigma\cap
\mathcal{O}$ is given by $\cK (x,y,z)$; here, we identify
$\Sigma_1,\Sigma_2\in \cA_1$ if they overlap on an open set
containing the singularity $p_0$.

Then, the map that sends each surface in $\cA_1$ to its limit unit
normal at the singularity provides a one-to-one correspondence
between $\cA_1$ and the class $\cA_2$ of regular, analytic, strictly
convex Jordan curves in $\S^2$.
 \end{teorema}

Let us point out that the $\Z_2$ factor appearing in the
correspondence of Theorem \ref{main} does not appear in Theorem
\ref{teo: embedded} by our choice of the \emph{canonical
orientation} for surfaces in $\cA_1$.

Theorem \ref{teo: embedded} generalizes \cite[Corollary 13]{GHM},
which covers the case $\cK= {\rm const.}$

\end{document}